\documentclass[a4paper,12pt,twoside]{amsart}

\usepackage{amsmath, amsthm, amsfonts, amssymb}
\usepackage{enumerate}
\usepackage{tikz}
\usepackage{tikz-cd}   
\usepackage{color}
\usepackage{geometry}
\usepackage{float}

\title[Equivariant Benjamini-Schramm convergence and $\ell^2$-multiplicities]
 {Equivariant Benjamini-Schramm convergence of simplicial complexes and $\ell^2$-multiplicities} 

\author[S. Kionke]{Steffen Kionke}
\author[M. Schr\"odl-Baumann]{Michael Schr\"odl-Baumann}

\address{Karlsruher Institut f\"ur Technologie \\ Fakult\"at f\"ur Mathematik \\
  Institut f\"ur Algebra und Geometrie \\ Englerstr.2 \\
  76131 Karlsruhe \\ Germany.}

\email{steffen.kionke@kit.edu}
\email{michael.schroedl@kit.edu}

\thanks{The research was funded by the Deutsche Forschungsgemeinschaft
  (DFG, German Research Foundation) -  338540207.}

\date{\today}

\theoremstyle{plain}
\newtheorem{theorem}{Theorem}
\newtheorem*{theorem*}{Theorem}
\newtheorem*{proposition*}{Proposition}
\newtheorem*{lemma*}{Lemma}
\newtheorem{lemma}[theorem]{Lemma}

\newtheorem*{corollary*}{Corollary}
\newtheorem{proposition}[theorem]{Proposition}

\theoremstyle{definition}
\newtheorem{definition}[theorem]{Definition}

\newtheorem{remark}[theorem]{Remark}
\newtheorem*{remark*}{Remark}
\newtheorem{example}[theorem]{Example}
\newtheorem*{example*}{Example}

\newtheorem*{question*}{Question}

\numberwithin{equation}{section}
\numberwithin{theorem}{section}

\DeclareMathOperator{\id}{Id}

\DeclareMathOperator{\ind}{ind}
\DeclareMathOperator{\Irr}{Irr}

\DeclareMathOperator{\im}{im}
\DeclareMathOperator{\tr}{tr}
\DeclareMathOperator{\pr}{pr}
\DeclareMathOperator{\Orb}{Orb}
\DeclareMathOperator{\spec}{spec}
\DeclareMathOperator{\St}{St} 

\newcommand{\normal}{\trianglelefteq}
\newcommand{\SC}[1]{\mathcal{SC}_{*}^D(#1)}
\newcommand{\SCo}[2]{\mathcal{SC}_{*}^{D}(#1,#2)}
\newcommand{\SCC}[1]{\mathcal{SC}_{**}^D(#1)}
\newcommand{\gcong}{\stackrel{G}{\cong}}
\newcommand{\CC}{C^{(2)}}
\newcommand{\HH}{H^{(2)}}
\providecommand{\bbN}{\mathbb{N}}
\providecommand{\bbR}{\mathbb{R}}

\providecommand{\bbQ}{\mathbb{Q}}
\providecommand{\bbZ}{\mathbb{Z}}

\providecommand{\bbC}{\mathbb{C}}

\renewcommand{\epsilon}{\varepsilon}
\renewcommand{\phi}{\varphi}

\begin{document}

\begin{abstract}
  We define a variant of Benjamini-Schramm convergence for finite simplicial complexes with the
  action of a fixed finite group $G$ which leads to the notion of random rooted
  simplicial $G$-complexes.  For every random rooted simplicial $G$-complex we define a
  corresponding $\ell^2$-homology and the $\ell^2$-multiplicity of an irreducible
  representation of $G$ in the homology.  The $\ell^2$-multiplicities generalize the
  $\ell^2$-Betti numbers and we show that they are continuous on the space of sofic random rooted
  simplicial $G$-complexes.
  In addition, we study induction of random rooted complexes and discuss the effect on $\ell^2$-multiplicities.
\end{abstract}

\maketitle

\section{Introduction}
Benjamini and Schramm \cite{Benjamini-Schramm} introduced the concept of random rooted
graphs as probability measures on the space of connected rooted graphs.
This allowed them to define convergence -- today known as \emph{Benjamini-Schramm convergence} -- of sequences
of finite graphs and to study the corresponding limit random rooted graphs.
It was realized by Aldous and Lyons \cite{Aldous-Lyons2007} that Benjamini-Schramm limits of finite graphs share a useful mass-transport property, called \emph{unimodularity}, which, when added to the definition of random rooted graphs, allows to extend several results from quasi-transitive graphs to this setting. These two insights provide the basis of a large number subsequent developments.

The basic idea of Benjamini-Schramm convergence is not specific to graphs and was, in particular,
generalized to simplicial complexes in the work of Elek \cite{Elek2010}, Bowen \cite{Bowen2015} and one of the authors \cite{schroedl}. The central result of \cite{schroedl} (extending similar results in \cite{Elek2010,Bowen2015}) is
that suitably defined $\ell^2$-Betti numbers of random rooted simplicial complexes are Benjamini-Schramm continuous (on the space of \emph{sofic} random rooted complexes).
For graphs a stronger result was proven in \cite{AbertThomVirag}.
Here we introduce an \emph{equivariant} Benjamini-Schramm convergence for simplicial complexes with the action of a fixed finite group; to our knowledge this is a new concept even for graphs. Then we define refinements of $\ell^2$-Betti numbers, namely \emph{$\ell^2$-multiplicities}, and prove a corresponding continuity result.

\subsection*{Statement of main results} We fix a finite group $G$ and consider the space $\SC{G}$ of isomorphism classes of rooted
simplicial $G$-complexes of vertex degree at most $D$. Here a rooted simplicial $G$-complex
$(K,o)$ consists of a simplical complex with an action of $G$ and a distinguished $G$-orbit $o$ of
vertices in $K$ which touches every connected component of $G$. A \emph{random rooted simplicial $G$-complex} $\mu$ is a unimodular probability measure on $\SC{G}$; for details we refer to Section~\ref{sec:random-rooted-complexes}.
For instance, picking a root uniformly at random in a finite simplicial $G$-complex defines a random rooted $G$-complex. Limits of laws of finite simplicial $G$-complexes are called \emph{sofic}; see Definition~\ref{def:sofic}.

 Let $K$ be a finite simplical
$G$-complex. The group $G$ acts on $K$ and thus acts on the homology groups $H_n(K,\bbC)$,
i.e., we have a finite dimensional representation of $G$ on $H_n(K,\bbC)$. This representation
decomposes as a direct sum of irreducible representations of $G$ and every irreducible
representation $\rho \in \Irr(G)$ occurs a finite number, say $m(\rho, H_n(K,\bbC))$, of times
in this decomposition. The number $m(\rho,H_n(K,\bbC))$ is called the \emph{multiplicity} of
$\rho$ in $H_n(K,\bbC)$.  Drawing from this, we define $\ell^2$-multiplicities in Section
\ref{sec:homology-and-mult}; these are nonnegative real numbers $m^{(2)}_n(\rho,\mu)$ for
every random rooted simplicial $G$-complex $\mu$ and every irreducible representation
$\rho \in \Irr(G)$. In fact, if $\mu_K$ is the law of a finite $G$-complex, then
\[ m_n^{(2)}(\rho,\mu_K) = \frac{m(\rho,H_n(K,\bbC))}{|K^{(0)}|};\]
see Example~\ref{ex:finite-complex-multiplicities}.
Our main result is the following approximation theorem.
\begin{theorem}\label{thm:approximation}
  Let $(\mu_k)_{k\in\bbN}$ be a sequence of random rooted simplicial $G$-complexes. If every $\mu_k$ is sofic and the sequence
  converges weakly to $\mu_\infty$, then 
  \[ \lim_{k\to \infty} m_n^{(2)}(\rho,\mu_k) = m_n^{(2)}(\rho, \mu_\infty) \]
  for every $n \in \bbN_0$ and every irreducible representation $\rho$ of $G$.
\end{theorem}

Along the way (in Section \ref{sec:induced-complexes})
we investigate induced $G$-complexes. Given a subgroup $H \leq G$ and a simplicial $H$-complex $L$, it is natural to construct
a simplicial $G$-complex $K = G \times_H L$ \emph{by inducing the action from $H$ to $G$}.
This can be promoted to an operation $\ind_H^G$ which takes random rooted $H$-complexes to random rooted $G$-complexes. We provide a criterion to decide whether a sequence of finite $G$-complexes converges to an induced random rooted $G$-complex; see Proposition~\ref{prop:criterion-convergence-induced}.
This is relevant, since the $\ell^2$-multiplicities of the induced random rooted complex $\ind_H^G(\mu)$ can be computed from the $\ell^2$-multiplicities of $\mu$ (provided $\mu$ is sofic).  As a special case, for $H=\{1\}$, we have the following result. 
\begin{theorem}\label{thm:multiplicities-Betti}
  Let $\mu$ be a sofic random rooted simplicial complex and let $G$ be a finite group.
  For all $\rho \in \Irr(G)$ and $n\in\bbN_0$ the following identity holds:
  \[ m_n^{(2)}(\rho,\ind_1^G(\mu)) = \frac{\dim_\bbC(\rho)}{|G|} b_n^{(2)}(\mu).\]
\end{theorem}

In the topological setting $L^2$-multiplicities of CW-complexes
have been defined and studied in \cite{Kionke18}. Our results are mainly complementary to the results therein. However, for towers of finite sheeted covering spaces of simplicial complexes
our results provide a different approach to \cite[Theorem 1.2]{Kionke18}.
In this setting, we think that our investigation of induced complexes is useful.
Combining Theorem~\ref{thm:multiplicities-Betti} with Proposition~\ref{prop:towers-induced} provides a
new perspective on the centralizer condition which remained a bit obscure in \cite[Proposition 5.5]{Kionke18}.

\section{Random rooted simplicial $G$-complexes}\label{sec:random-rooted-complexes}
We will use the language of simplicial complexes as introduced in many standard texts; for instance \cite{Spanier81}. The $n$-skeleton of a simplicial complex $K$ is denoted by $K^{(n)}$. The number of $1$-simplices containing a vertex $x \in K^{(0)}$ is called the \emph{degree} of $x$. The supremum over the degrees of all vertices of $K$ will be called the \emph{vertex degree} of $K$.
To avoid some technical issues we will consider simplicial complexes whose vertex degree is bounded above by some large constant $D >0$.
For a simplicial complex $K$, a non-empty set of vertices $V\subseteq K^{(0)}$ and $r>0$ we define
$B_r(K,V)$ to be the full subcomplex on the set of vertices of path distance
at most $r$ from a vertex in $V$.

We fix a finite group $G$.
A \emph{rooted simplicial $G$-complex} is a pair $(K,o)$ consisting of a simplicial complex $K$ with $G$-action and a $G$-orbit $o$ of vertices of $K$ such that every connected component of $K$ contains at least one vertex in $o$. Two rooted simplicial $G$-complexes $(K,o)$, $(L,o')$ are isomorphic if there is a simplicial isomorphism $f\colon K\to L$ such that $f(o)=o'$ and $g\cdot f(x)=f(g\cdot x)$ for all $x\in K^{(0)}$ and $g\in G$; in this case we write
$(K,o) \gcong (L,o')$.

We denote by $\SC{G}$ \emph{the space of isomorphism classes of rooted simplicial $G$-complexes} with vertex degree bounded by $D$.
The following defines a metric on $\SC{G}$:
\[d([K,o],[L,o'])=\inf\left\{\frac{1}{2^r}\;\middle|\; B_r(K,o)\gcong B_r(L,o')\right\}\]
for $[K,o],[L,o']\in\SC{G}$
Note that subcomplexes of the form $B_r(K,o)$ are stable under the $G$-action.
With the induced topology $\SC{G}$ is a compact totally disconnected space. For every finite rooted simplicial $G$-complex $\alpha$ of
radius at most $r$, we obtain an open subset $U(\alpha,r) \subseteq \SC{G}$
consisting of all isomorphism classes of rooted simplicial $G$-complexes $[K,o]$ such that $B_r(K,o) \gcong \alpha$. The sets $U(\alpha,r)$ are open and compact and provide a basis for the topology.

Given two rooted simplicial $G$-complexes $(K,o)$ and $(L,o')$, such that $o$ and $o'$ are not isomorphic as $G$-sets, then $d([K,o],[L,o']) = \infty$.
Let $\Orb(G)$ denote the finite set of isomorphism classes of transitive $G$-sets.
The space $\SC{G}$ splits into a disjoint union
\[\SC{G}=\bigsqcup_{X \in \Orb(G)} \SCo{G}{X},\]
where  $\SCo{G}{X}$ is the subspace of all $[K,o]$ where $o$ belongs to $X$.

In the same manner we define the space $\SCC{G}$ of isomorphism classes of doubly rooted simplicial $G$-complexes $(K,o,o')$ consisting of a simplicial $G$-complex $K$ and two orbits $o$ and $o'$ of vertices which touch every connected component.
\begin{definition}
  A \emph{random rooted simplicial $G$-complex} is a \emph{unimodular} probability measure $\mu$ on $\SC{G}$, where unimodular means
\begin{align*}
\int_{\SC{G}}\sum_{x \in K^{(0)}}f([K,o,Gx])d\mu([K,o])=\int_{\SC{G}}\sum_{x \in K^{(0)}}f([K,Gx,o])d\mu([K,o])
\end{align*}
for all Borel measurable functions $f\colon \SCC{G}\to \bbR_{\ge 0}$.
\end{definition}
\begin{remark}
  A weak limit of  random rooted simplicial $G$-complexes is again a random rooted simplicial $G$-complex. This can be deduced by observing
  that a probability measure $\mu$ on $\SC{G}$ is unimodular
  if and only if for all $n \in \bbN$ and all \emph{continuous} functions $f \colon \SCC{G} \to \bbR$
  the equality
  \[\int_{\SC{G}}\sum_{{\substack{x \in K^{(0)} \\ d(x,o) \leq n}}}f([K,o,Gx])d\mu([K,o])=\int_{\SC{G}}\sum_{\substack{x \in K^{(0)} \\ d(x,o) \leq n}}f([K,Gx,o])d\mu([K,o]) \]
  holds, where the sum now runs over the vertices of path distance at most $n$ from the root orbit.
  This can be verified using the theorem of monotone convergence and the monotone class theorem.
\end{remark}

\begin{example}\label{example-finiterrsc}
Let $K$ be a finite simplicial $G$-complex. Then there is a unique random rooted simplicial $G$-complex which is fully supported on isomorphism classes of the form $[K,o]$; it is given by
\begin{align*}
\mu^G_K:=\sum_{x\in K^{(0)}}\frac{\delta_{[K,Gx]}}{|K^{(0)}|},
\end{align*}
where $\delta_{[K,Gx]}$ denotes the Dirac measure of the point $[K,Gx]$. That $\mu^G_K$ is a probability measure is obvious, we only have to verify that it is unimodular:
\begin{align*}
\int_{\SC{G}}\sum_{x\in L^{(0)}}f([L,o,Gx])d\mu_K^G([L,o])&=\sum_{y\in K^{(0)}}\frac{1}{|K^{(0)}|}\sum_{x\in K^{(0)}}f([K,Gy,Gx]) \\
&=\sum_{x\in K^{(0)}}\frac{1}{|K^{(0)}|}\sum_{y\in K^{(0)}}f([K,Gy,Gx])\\
&=\int_{\SC{G}}\sum_{y\in L^{(0)}}f([L,Gy,o])d\mu^G_K([L,o]).
\end{align*}
\end{example}
\begin{definition}\label{def:sofic}
A sequence $(K_n)_n$ of finite simplicial $G$-complexes (of vertex degree bounded by $D$) converges \emph{Benjamini-Schramm} if the weak limit $\lim_{n\to\infty}\mu_{K_n}^G$ exist. A random rooted simplicial $G$-complex on $\SC{G}$ which is the Benjamini-Schramm limit of a sequence of finite simplicial $G$-complexes is called  \emph{sofic}.
\end{definition}

\section{Induction of simplicial complexes}\label{sec:induced-complexes}
Let $G$ be a finite group and let $H \leq G$ be a subgroup.
Given an $H$-set $Y$ one can construct the induced $G$-set $G \times_H Y$.
The set
\begin{equation*}
  G\times_H Y = (G \times Y) /\sim
\end{equation*}
is obtained by forming the quotient of $G \times Y$ under the equivalence relation
$(g,y) \sim (gh,h^{-1}\cdot y)$ for all $g \in G$, $y \in Y$ and $h \in H$.
We write $\lfloor g,y \rfloor$ to denote the equivalence class of $(g,y)$.
Clearly $g_2 \cdot \lfloor g_1,y \rfloor = \lfloor g_2g_1, y\rfloor$ defines an action of $G$ on $G\times_H Y$.

A simplicial $H$-complex $K$ gives rise to a simplicial $G$-complex $G \times_H K$ by induction.
The vertices of $G \times_H K$  are the elements of $G \times_H K^{(0)}$. For every simplex $\sigma \subseteq K^{(0)}$ of $K$ and every $g\in G$ we define a simplex
$\sigma_g = \{\lfloor g, x \rfloor \mid x \in \sigma \}$ of $G\times_H K$. As a simiplicial complex (without $G$-action) $G \times_H K$ is isomorphic
to a disjoint union of $|G/H|$ copies of $K$. In particular, induction does not alter the vertex degree.
\begin{lemma}\label{lem:continuity-of-ind}
The function
   $\ind_H^G \colon \SC{H} \to \SC{G}$
 which maps $[K,o]$ to $[G\times_H K, G\lfloor 1, o\rfloor]$
 is continuous. In particular, the push-forward of measures with $\ind_H^G$ is weakly continuous.
\end{lemma}
\begin{proof}
 We observe that the ball of radius $r$ in $G\times_H K$ around $G\lfloor 1, o\rfloor$ is
 isomorphic to $G \times_H B_r(K,o)$.
 This implies that \[d\bigl(\ind_H^G([K,o]),\ind_H^G([L,o'])\bigr) \leq d([K,o],[L,o'])\] and proves the assertion.
\end{proof}

 We use this to define induction of random rooted simplicial complexes.
 \begin{lemma}
   Let $\mu$ be a random rooted simplicial $H$-complex. The push-forward measure $\ind_H^G(\mu)$
   is a random rooted simplicial $G$-complex.
 \end{lemma}
 \begin{proof}
   The push-forward preserves the total mass, so $\ind_H^G(\mu)$ is a probability measure.
   It remains to verify unimodularity.
   Recall the general transformation rule \cite[\S 6]{Bourbaki-Integration}
   \[\int_{\SC{G}} t(z) d\ind_H^G(\mu)(z) = \int_{\SC{H}} t(\ind_H^G(w)) d\mu(w)\]
   for all measurable nonnegative functions $t$ on $\SC{G}$.
   We obtain for all measurable $f \colon \SCC{G} \to \bbR_{\geq 0}$
   \begin{align*}
     \int_{\SC{G}} \sum_{x \in L^{(0)}} &f(L,o, Gx) d\ind_H^G(\mu)([L,o])\\
     \stackrel{(1)}{=} &\int_{\SC{H}} \sum_{x \in G \times_H K^{(0)}} f(G\times_H K, Go, Gx) d\mu([K,o])\\
     =  &\int_{\SC{H}} |G/H|\sum_{y \in K^{(0)}} f(G\times_H K, Go, G\lfloor 1,y \rfloor ) d\mu([K,o])\\
     \stackrel{\text{(2)}}{=}   &\int_{\SC{H}} |G/H|\sum_{y \in K^{(0)}} f(G\times_H K, G\lfloor 1,y \rfloor, Go ) d\mu([K,o])\\
     \stackrel{(3)}{=}
     &\int_{\SC{G}} \sum_{x \in L^{(0)}} f(L, Gx, o) d\ind_H^G(\mu)([L,o])\\
   \end{align*}
   where we use the transformation rule in steps (1) and (3), and the unimodularity of $\mu$ in step (2).
 \end{proof}

 The following criterion is useful to show that a sequence of finite simplicial $G$-complexes converges to an induced random rooted simplicial $G$-complex.
 \begin{proposition}\label{prop:criterion-convergence-induced}
   Let $(K_n)_n$ be a sequence of finite simplicial $G$-complexes with vertex degree bounded by $D$.
   Assume that the sequence of random rooted simplical $H$-complexes
   $(\mu_{K_n}^H)_{n}$ converges to a random rooted simplicial $H$-complex $\mu_{\infty}$ for some subgroup $H \leq G$.
   Then $(\mu_{K_n}^G)_n$ converges to  $\ind_H^G(\mu_\infty)$ on $\SC{G}$ if and only if
   \begin{equation}\label{eq:points-moved}
     \lim_{n\to\infty} \frac{|E(K_n,g,C)|}{|K_n^{(0)}|} = 0 
     \end{equation}
     for all $C>0$ and all $g \in G \setminus H$ where
     $E(K,g,C) = \{ x \in K^{(0)} \mid d(x,gx) \leq C\}$.
 \end{proposition}
 \begin{proof}
   Define $E(K,C) = \bigcup_{g \in G \setminus H} E(K,g,C)$.

   Assume that equation \eqref{eq:points-moved} holds for all $g \in G \setminus H$ and all $C>0$.
   Let $r>0$ be given. We want to verify that for all $x \in K_n^{(0)} \setminus E(K_n,2r+1)$ the ball of radius $r$ around $Gx$ in $K_n$ is isomorphic to $G \times_H B_r(K_n,Hx)$.
    This amounts to showing that
    \[ B_r(K_n,Hx) \cap g B_r(K_n,Hx) = \emptyset\]
    for all $g \in G \setminus H$ and that there is no edge between these two sets.
   Suppose that there is an element in the intersection or an edge between $B_r(K_n,Hx)$ and $g B_r(K_n,Hx)$. In both cases we can
  find $h,h' \in H$ such that $d(hx, gh'x) \leq 2r+1$. However this implies $x \in E(K_n,2r+1)$
  since  $d(x,h^{-1}gh'x) \leq 2r+1$ and $h^{-1}gh' \not\in H$.

  Let $\alpha$ be a rooted simplicial $G$-complex of radius at most $r$.
    Let $\epsilon > 0$ and take $n$ sufficiently large such that
    $|E(K_n,2r+1)|< |K_n^{(0)}| \epsilon$.
    The inverse image $V = (\ind_H^G)^{-1}(U(\alpha,r)) \subseteq \SC{H}$ is a finite (possibly empty) union of
    sets of the form $U(\alpha',r)$; thus it is open and compact.
  The weak convergence $\mu_{K_n}^H \stackrel{w}{\longrightarrow} \mu_\infty$ shows that for all sufficiently large $n$ the inequality  $|\mu_{\infty}(V) - \mu_{K_n}^H(V)| < \epsilon$ holds.
  Moreover, by the observation above
  \[ \left|\mu^G_{K_n}(U(\alpha,r)) - \mu_{K_n}^H(V)\right| \leq   \frac{|E(K_n,2r+1)|}{|K_n^{(0)}|} < \epsilon.\]
  As $\alpha$ was arbitrary, we deduce the convergence  $\mu_{K_n}^G \stackrel{w}{\longrightarrow} \ind_H^G(\mu_\infty)$.
  
  \medskip

  Conversely, suppose that the sequence $(\mu_{K_n}^G)_n$ converges to $\ind_H^G(\mu_\infty)$.
  Let $g \in G\setminus H$, $C >0$ and $x \in E(K_n,g,C)$.
  The ball $B_{C}(K_n,Hx)$ contains a path from $x$ to $gx$  and
  hence it is not isomorphic to a simplicial complex induced from $H$.
  By assumption the limit $\lim_{n\to \infty} \mu_{K_n}^G$ is supported on
  induced complexes and thus \eqref{eq:points-moved} is satisfied.
\end{proof}

\begin{example}[Sierpinski's triangle with rotation]\label{ex:Sierpinski1}
  We describe a sequence $(T_n)_n$ of two-dimensional simplicial complexes which occur in the construction of the fractal Sierpinski triangle.
  It appeared to us that the example becomes clearer if we describe the geometric realizations of the $T_n$ as subsets of $\bbR^2$ instead of working with the abstract simplicial complexes.
  Let $e_1 = (1, 0) \in \bbR^2$ and let $e_2 = \frac{1}{2}(1,\sqrt{3}) \in \bbR^2$.
  The points $0$, $e_1$ and $e_2$ are the vertices of an equilateral triangle $T_0$ with sides of length $1$; we consider $T_0$ to  be a $2$-simplex.
  We define inductively
  \[ T_{n+1} = T_n \cup (T_n+2^n e_1) \cup (T_n+2^n e_2).\]
    
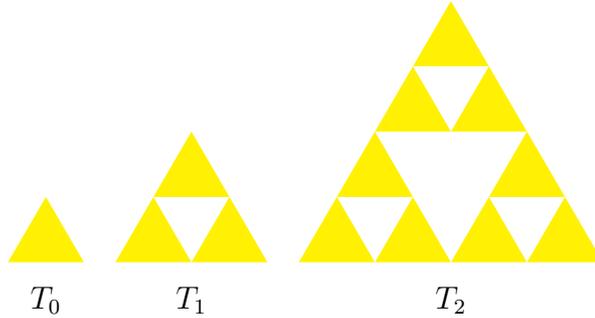
\begin{figure}[H]
  \begin{tikzpicture}
    \fill[fill=yellow]
    (0,0) node {}
 -- (1,0) node {}
 -- (0.5,0.86603) node {};
\node[align=left] at (0.5,-0.5) {$T_0$};
\end{tikzpicture}
\begin{tikzpicture}
  \foreach \a in {0,1}
  {
    \foreach \b in {\a,...,1}
    {
      \fill[fill=yellow]
         (\a * 0.5 + \b * 0.5, \b * 0.86603 - \a * 0.86603) node {}
      -- (\a * 0.5 + \b * 0.5 + 1,  \b * 0.86603 - \a * 0.86603) node {}
      -- (\a * 0.5  + \b* 0.5 + 0.5, \b * 0.86603 - \a * 0.86603+ 0.86603) node {};
    }
  }
 \node[align=left] at (1,-0.5) {$T_1$};
\end{tikzpicture}
\begin{tikzpicture}
  \foreach \a in {0,1}
  {
    \foreach \b in {\a,...,1}
    {
      \fill[fill=yellow]
         (\a * 0.5 + \b * 0.5, \b * 0.86603 - \a * 0.86603) node {}
      -- (\a * 0.5 + \b * 0.5 + 1,  \b * 0.86603 - \a * 0.86603) node {}
      -- (\a * 0.5  + \b* 0.5 + 0.5, \b * 0.86603 - \a * 0.86603+ 0.86603) node {};
    }
  }
   \foreach \a in {0,1}
  {
    \foreach \b in {\a,...,1}
    {
      \fill[fill=yellow]
         (2+\a * 0.5 + \b * 0.5, \b * 0.86603 - \a * 0.86603) node {}
      -- (2+\a * 0.5 + \b * 0.5 + 1,  \b * 0.86603 - \a * 0.86603) node {}
      -- (2+\a * 0.5  + \b* 0.5 + 0.5, \b * 0.86603 - \a * 0.86603+ 0.86603) node {};
    }
  }
   \foreach \a in {0,1}
  {
    \foreach \b in {\a,...,1}
    {
      \fill[fill=yellow]
         (1+\a * 0.5 + \b * 0.5, \b * 0.86603 - \a * 0.86603 + 1.7321) node {}
      -- (1+\a * 0.5 + \b * 0.5 + 1,  \b * 0.86603 - \a * 0.86603 + 1.7321) node {}
      -- (1+\a * 0.5  + \b* 0.5 + 0.5, \b * 0.86603 - \a * 0.86603+ 0.86603 + 1.7321) node {};
    }
  }
  \node[align=left] at (2,-0.5) {$T_2$};
\end{tikzpicture}
\end{figure}
  By induction it is easy to verify that $T_n$ is a simplicial complex with
  $\frac{3^{n+1}+3}{2}$ vertices, $3^{n+1}$ edges and $3^n$ $2$-simplices.
  The vertex degree of $T_n$ is $4$ for all $n\geq 1$.
  The three vertices of degree $2$ will be called the \emph{corners} of $T_n$. The distance between two corners of $T_n$ is $2^n$.

  \medskip

  \noindent\emph{Claim:}
  The sequence $(T_n)_n$ converges Benjamini-Schramm to a random rooted simplicial complex $\tau_S$.

  \smallskip

  Let $r >0$ and let $\alpha$ be a finite rooted simplicial complex of radius at most $r$.
  Take $m$ so large that $2^{m-1} > r$. Then any $r$-ball in $T_m$ contains at most one of the three corners of $T_m$.
  Let $N(k,\alpha)$ denote the number of vertices $v$ in $T_{m+k}$
  such that the ball of radius $r$ around $v$ is isomorphic to $\alpha$.
  We observe that
  \[ N(k+1,\alpha) = 3 N(k,\alpha) + c_\alpha\]
  for all $k \geq 0$ for some constant $c_\alpha \in \bbZ$.
  Indeed, $r$-balls around vertices of distance at least $r$ from one of the corners in $T_{m+k}$ occur exactly $3$-times in $T_{m+k+1}$. In the small set of vertices which lie
  close to a corner of $T_{m+k}$, we always see two copies of $T_m$ being glued at a corner.
  Which shows that the effect of this operation does not depend on $k$; compare Figure \ref{fig:small-balls}.
  \begin{figure}[H]\label{fig:small-balls}
  \resizebox{0.4\textwidth}{!}{
%
%

\begin{tikzpicture}
  \foreach \a in {0,1}
  {
    \foreach \b in {\a,...,1}
    {
      \fill[fill=yellow]
         (\a  + \b , \b * 1.732 - \a * 1.732) node {}
      -- (\a  + \b  + 2,  \b * 1.732 - \a * 1.732) node {}
      -- (\a   + \b + 1, \b * 1.732 - \a * 1.732+ 1.732) node {};
    }
  }
   \foreach \a in {0}
  {
    \foreach \b in {\a}
    {
      \fill[fill=yellow]
         (4+\a  + \b , \b * 1.732 - \a * 1.732) node {}
      -- (4+\a  + \b  + 2,  \b * 1.732 - \a * 1.732) node {}
      -- (4+\a   + \b + 1, \b * 1.732 - \a * 1.732+ 1.732) node {};
    }
  }
   \foreach \a in {0}
  {
    \foreach \b in {\a}
    {
      \fill[fill=yellow]
         (2+\a  + \b , \b * 1.732 - \a * 1.732+3.4641) node {}
      -- (2+\a  + \b  + 2,  \b * 1.732 - \a * 1.732+3.4641) node {}
      -- (2+\a   + \b + 1, \b * 1.732 - \a * 1.732+ 1.732+3.4641) node {}; node {};
    }
  }
  \node[align=left] at (1,0.6) {\tiny{$T_m$}};
  \node[align=left] at (2,2.4) {\tiny{$T_m$}};
  \node[align=left] at (3,0.6) {\tiny{$T_m$}};
  \node[align=left] at (3,1.7 + 2.4) {\tiny{$T_m$}};
  \node[align=left] at (5,0.6) {\tiny{$T_m$}};
  \node[align=left] at (4,-0.5) {{$T_{m+k}$}};
  
  \draw[dotted] (3,1.732+3.4641) -- (3.5,1.732+3.4641+0.865);
  \draw[dotted] (6,0) -- (6.5,0);

  \fill[green,opacity=0.3] (2.4,2.4249 + 1.732) arc(60:-120:0.5);
  \fill[green,opacity=0.3](3.5,0) arc (180:0:0.5);
\end{tikzpicture}
%
}
  \caption{Small balls touch at most two copies of $T_m$}
  \end{figure}
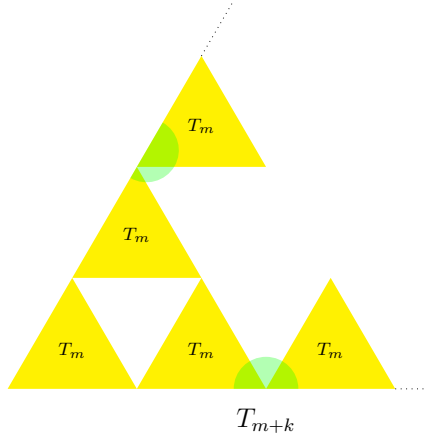
  Now it follows from a short calculation that $\Bigl(\frac{|N(k,\alpha)|}{|T_{m+k}^{(0)}|}\Bigr)_{k}$
  is a Cauchy sequence. Since $r$ and $\alpha$ were arbitrary, we conclude that the sequence $(T_n)_n$ converges in the sense of Benjamini-Schramm.

  \medskip

  Now we introduce an action of the finite cyclic group $G = \langle \sigma \rangle$ of order
  $3$. We let $\sigma$ act by rotation of $2\pi/3$ around the barycenter
  $c_n = 2^{n-1} (1, \sqrt{3}^{-1})$ of $T_n$. All vertices of $T_n$ have Euclidean distance
  at least $\frac{2^{n-2}}{\sqrt{3}}$ from the barycenter. Thus every vertex is moved by an
  Euclidean distance of at least $2^{n-2}$ under the non-trivial rotations $\sigma$ and
  $\sigma^2$, which in particular also holds for the path distance in $T_n$. Proposition
  \ref{prop:criterion-convergence-induced} implies that the sequence $(T_n)_n$ of simplicial
  $G$-complexes converges to the induced random rooted simplicial $G$-complex
  $\ind_{1}^G(\tau_S)$.  Roughly speaking, the sequence converges to three copies of the
  Sierpinski triangle which are permuted cyclically by $G$.
\end{example}

\subsection{Towers of finite sheeted covering spaces}
In this section we discuss a prominent family of examples of Benjamini-Schramm convergent sequences: towers of finite sheeted covering spaces.

Let $\Gamma$ be a group and let $K$ be simplicial complex of vertex degree at most $D$.
Assume that $\Gamma$ acts \emph{simplicially} on $K$, this means that an element $\gamma \in \Gamma$ stabilizes a simplex of $K$ if and only if it stabilizes all of its vertices. Recall that this condition can always be achieved by passing to the barycentric subdivision of $K$. We assume further that the action is \emph{proper} and \emph{cocompact}, i.e., every vertex has a finite stabilizer and there are only finitely many orbits of vertices.

Let $G \leq \Gamma$ be a finite subgroup. For every normal subgroup $N \normal \Gamma$ the quotient simplicial complex $K/N$ carries a $G$-action.
Suppose that $\Gamma$ is residually finite and let $(N_n)_{n\in\bbN}$ be a descending chain of finite index normal subgroups of $\Gamma$ with $\bigcap_{n \in \bbN} N_n = \{1\}$. It is well-known (cf. \cite[Example 19]{schroedl}) that the sequence $(K/N_n)_n$ of finite simplicial complexes (without $G$-action) converges to the random rooted simplicial complex
\[ \frac{1}{w(\Gamma)} \sum_{x \in \mathcal{F}}|\St_\Gamma(x)|^{-1} \delta_{[K,x]}  \]
where $\mathcal{F}$ is a fundamental domain for the action of $\Gamma$ on $K^{(0)}$ and
$w(\Gamma) = \sum_{x \in \mathcal{F}}  |\St_\Gamma(x)|^{-1}$. This measure does not depend on the choice of the fundamental domain.
The purpose of this section is to describe the limit taking the $G$-action into account.

\medskip

We say that an element $\gamma \in \Gamma$ is \emph{FC} if it has a finite conjugacy class, i.e.,
$|\Gamma:C_\Gamma(\gamma)|<\infty$, where $C_\Gamma(\gamma)$ is the centralizer of $\gamma$.
Consider the subgroup $H = \{ g \in G \mid g \text{ is FC in } \Gamma\} \leq G$ of FC-elements which lie in $G$.
Let $\Gamma_0 \leq_{f.i.} \Gamma$ be a finite index subgroup which satisfies
\[ \Gamma_0 \subseteq \bigcap_{h\in H} C_\Gamma(h).\]

\begin{lemma}
  Let $\mathcal{F}_0 \subseteq K^{(0)}$ be a fundamental domain for the action of $\Gamma_0$ on $K^{(0)}$.
  The measure
  \[ \mu_K^H = \frac{1}{w(\Gamma_0)} \sum_{x \in \mathcal{F}_0} |\St_{\Gamma_0}(x)|^{-1}\delta_{[K,Hx]}\]
  on $\SC{H}$ is unimodular and does not depend on the choices of $\Gamma_0$ and $\mathcal{F}_0$.
\end{lemma}
\begin{proof}
  Observe that any element
  $\gamma \in \Gamma_0$ commutes with all $h \in H$ and thus defines an isomorphism between
  $(K,Hx)$ and $(K,H\gamma x)$ as simplicial $H$-complexes. This shows that the measure is independent of the fundamental domain.

  In order to verify that $\mu_K^H$ does not depend on $\Gamma_0$, it is sufficient to show that we can replace $\Gamma_0$ by some finite index normal subgroup $\Gamma_1 \normal_{f.i.} \Gamma_0$.
  Let $\mathcal{F}_1$ be a fundamental domain for $\Gamma_1$.
  We obtain
  \begin{align*}
    &\sum_{x \in \mathcal{F}_1} |\St_{\Gamma_1}(x)|^{-1} \delta_{[K,Hx]}
    = \sum_{y \in \mathcal{F}_0} |\St_{\Gamma_0}(y)|^{-1}\sum_{\substack{\gamma \in \Gamma_0 \\ \gamma y \in \mathcal{F}_1}}  |\St_{\Gamma_1}(\gamma y)|^{-1} \delta_{[K,H\gamma y]}\\
   = &\sum_{y \in \mathcal{F}_0}|\St_{\Gamma_0}(y)|^{-1} \delta_{[K,H y]}\sum_{\substack{\gamma \in \Gamma_0 \\ \gamma y \in \mathcal{F}_1}} |\St_{\Gamma_1}(y)|^{-1}
    = |\Gamma_0:\Gamma_1| \sum_{y \in \mathcal{F}_0}|\St_{\Gamma_0}(y)|^{-1} \delta_{[K,H y]}
  \end{align*}
  and, from a similar calculation, also $w(\Gamma_1) = |\Gamma_0:\Gamma_1|  w(\Gamma_0)$.

  It remains to show that $\mu_K^H$ is unimodular.
  Let $f \colon \SCC{H} \to \bbR_{\geq 0}$ be a measurable function.
  Unimodularity follows from a short calculation.
  \begin{align*}
    \int_{\SC{H}} \sum_{ y \in L^{(0)}}& f(L,o,Hy) d\mu_K^H([L,o])\\
    &= \frac{1}{w(\Gamma_0)} \sum_{x \in \mathcal{F}_0} |\St_{\Gamma_0}(x)|^{-1} \sum_{y \in K^{(0)}} f(K,Hx,Hy)\\
    &= \frac{1}{w(\Gamma_0)} \sum_{x \in \mathcal{F}_0} |\St_{\Gamma_0}(x)|^{-1} \sum_{y \in \mathcal{F}_0}\sum_{\gamma \in \Gamma_0} |\St_{\Gamma_0}(y)|^{-1} f(K,Hx,H\gamma y)\\
                  &= \frac{1}{w(\Gamma_0)} \sum_{x,y \in \mathcal{F}_0} |\St_{\Gamma_0}(x)|^{-1}|\St_{\Gamma_0}(y)|^{-1}\sum_{\gamma \in \Gamma_0}  f(K,H\gamma^{-1}x,H y) \\
   &= \dots  =  \int_{\SC{H}} \sum_{ x \in L^{(0)}} f(L,Hx,o) d\mu_K^H([L,o]) \qedhere
  \end{align*}
\end{proof}
\begin{proposition}\label{prop:towers-induced}
  Let $G \leq \Gamma$ be a finite subgroup and let $H \leq G$ be the subgroup of FC-elements for $\Gamma$.
  Let $(N_n)_n$ be a descending chain of finite index normal subgroups in $\Gamma$ with
  $\bigcap_{n \in \bbN} N_n = \{1\}$.
  The sequence of simplicial $G$-complexes $(K/N_n)_n$ converges to the random rooted simplicial
  complex $\mu_K^G := \ind_H^G(\mu_K^H)$.
\end{proposition}
\begin{proof}
  The proof consists of two steps. First we show that $(K/N_n)_n$ converges as a sequence of $H$-complexes to $\mu_K^H$ (Claims 1 and 2) and in the second step (Claim 3)  we apply
  Proposition~\ref{prop:criterion-convergence-induced}.

  \medskip
  
  \noindent\emph{Claim 1:} Let $r>0$ be fixed. For all sufficiently large $n$ and all $x \in K^{(0)}$, the $r$-ball $B_r(K,Hx)$ in $K$ and the $r$-ball $B_r(K/N_n,HN_nx)$ in $K/N_n$ are isomorphic as $H$-complexes.

  \smallskip

  Let $\Gamma_0 \leq \Gamma$ be as above and let $\mathcal{F}_0$ be a fundamental domain for
  $\Gamma_0$ acting on $K^{(0)}$.  The action is proper, the sets $\mathcal{F}_0$ and $H$ are
  finite and the vertex degree of $K$ is bounded, hence the set $S$ of elements
  $\gamma \in \Gamma$ such that
  \begin{equation}\label{eq:nonempty-intersection}
    B_{r+1}(K,Hx_0) \cap \gamma B_{r+1}(K,Hx_0) \neq \emptyset
  \end{equation}
  for some $x_0 \in \mathcal{F}_0$ is finite.
  
  Take $n \in \bbN$ so large that $S \cap N_n = \{1\}$.
  Then for all $x_0 \in K^{(0)}$ and $\gamma \in N_n$ property \eqref{eq:nonempty-intersection} implies that $\gamma = 1$.
  Indeed,
  find $\gamma_0 \in \Gamma_0$ with $\gamma_0 x_0 \in \mathcal{F}_0$ then multiplication with $\gamma_0$ yields
  $B_{r+1}(K,H\gamma_0x_0) \cap \gamma_0\gamma\gamma_0^{-1} B_{r+1}(K,H\gamma_0x_0) \neq \emptyset$.
  We deduce  $\gamma_0 \gamma \gamma_0^{-1} \in S \cap N_n =\{1\}$.
   In particular, the quotient map takes the vertices of the $r$-ball $B_r(K,Hx)$
  injectively to the $r$-ball $B_r(K/N_n,HN_nx)$.  We have to verify that every simplex in
  $B_r(K/N_n,HN_nx)$ lifts to a unique simplex in $B_r(K,Hx)$.  Let $\sigma$ be a simplex
  in $B_r(K/N_n,HN_nx)$ and let $\tilde{\sigma}$ be a lift in $K$ such that at least one
  vertex lies in $B_r(K,Hx)$. As a consequence $\tilde{\sigma}$ is a simplex in
  $B_{r+1}(K,Hx)$.  Let $y$ be any vertex of $\tilde{\sigma}$. There is an element $k \in N_n$ such that
  $d(ky,Hx) \leq r$. This means that $B_{r+1}(K,Hx) \cap kB_{r+1}(K,Hx) \neq \emptyset$ and
  shows that $k=1$. In particular, the simplex $\tilde{\sigma}$ lives in $B_r(K,Hx)$.

  \medskip

  \noindent\emph{Claim 2:} The sequence $\mu_{K/N_n}^H$ converges to $\mu_K^H$.
  
  \smallskip
  
  Let $r >0$ and let $\alpha$ be a rooted simplicial $H$-complex of radius at most $r$.
  Let $n \in \bbN$ sufficiently large such that $N_n$ acts freely on $K$ and so that Claim 1 applies.
  In addition, we may take $\Gamma_0 \le N_n$; the action of $\Gamma_0$ is also free.
  Now every point in $K/N_n$ is covered by exactly $|N_n : \Gamma_0|$ points in $\mathcal{F}_0$
  and we deduce
  \begin{align*}
    \mu_{K}^H(U(\alpha,r))
    &= \frac{|\{ x \in \mathcal{F}_0 \mid B_r(K,Hx) \stackrel{H}{\cong} \alpha\}|}{|\mathcal{F}_0|}\\
    &= \frac{|\{ x \in K^{(0)}/N_n \mid B_r(K,HN_nx) \stackrel{H}{\cong} \alpha\}|}{|K^{(0)}/N_n|} = \mu_{K/N_n}^H(U(\alpha,r))
  \end{align*}

  \medskip

  \noindent\emph{Claim 3:} \eqref{eq:points-moved} holds
  for all $C>0$ and all $g \in G \setminus H$.

  \smallskip

  Let $i \in \bbN$ be chosen so that $N_i$ acts freely on $K$ and
  let $\mathcal{F}$ be a fundamental domain for the action of $N_i$ on $K^{(0)}$.
  Let $Z \subseteq \Gamma$ be the finite set of elements $\gamma \in \Gamma$ such that
  $d(\gamma x, x) \leq C$ for some $x \in \mathcal{F}$.
  For $n \geq i$ the vertices of $K/N_n$ correspond bijectively to $N_i/N_n \times \mathcal{F}$.

  Take $x \in K^{(0)}$ and write $\bar{x} = N_nx \in K^{(0)}/N_n$. Suppose that
  $\bar{x} \in E(K/N_n,g,C)$; i.e., there is $\gamma_n \in N_n$ with
  $d(gx,\gamma_n x) \leq C$.  There is a unique $x_0 \in \mathcal{F}$ and an element
  $\gamma_i \in N_i$ satisfying $x = \gamma_i x_0$. This shows that
  $d(\gamma_i^{-1} \gamma_n^{-1}g\gamma_ix_0,x_0) \leq C$ and so
  $\gamma_i^{-1}g\gamma_i \in ZN_n$.  How many elements has the finite set
  $e_n(g,Z) = \{k \in N_i/N_n \mid k^{-1}gk \in ZN_n/N_n\}$?
  Clearly, its cardinality is bounded above by $|Z| \cdot |C_{N_i/N_n}(gN_n)|$.
  The element $g \in G$ has an infinite conjugacy class in $\Gamma$ and
  thus
  \[ \lim_{n\to \infty} \frac{|C_{\Gamma/N_n}(gN_n)|}{|\Gamma:N_n|} = 0;\]
  see the proof of \cite[Lemma 4.12]{Kionke18}. We deduce that
  \[ \lim_{n\to\infty} \frac{|E(K/N_n,g,C)|}{|K^{(0)}/N_n|} \leq  \lim_{n\to\infty} \frac{|e_n(g,Z)|}{|K^{(0)}/N_n|} \leq  \lim_{n\to\infty} \frac{|Z|\cdot |C_{\Gamma/N_n}(gN_n)|}{|N_i : N_n| |K^{(0)}/N_i|}  = 0.\]
  This proves the last claim and Proposition~\ref{prop:criterion-convergence-induced}
  completes the proof.
\end{proof}

\section{Homology and $\ell^2$-multiplicities of random rooted complexes}\label{sec:homology-and-mult}
\subsection{The homology of a random rooted complex}
In order to define $\ell^2$-multiplicities we introduce the $\ell^2$-homology of random rooted simplicial $G$-complexes. To this end, we will construct a chain complex for each probability measure on $\SC{G}$.
We begin with a technical ingredient which allows us to pick a representative for each isomorphism class $[K,o]\in \SC{G}$ of rooted simplicial $G$-complexes in a measurable way.

Let 
\[\bbN_G=\bigsqcup_{X\in \Orb(G)} \bbN_0\times X\]
and let $\Delta^D(\bbN_G)$ be the simplicial $G$-complex consisting of all finite nonempty subsets of
$\bbN_G$ with at most $D+1$ elements. The action of $G$ is defined via the second coordinate.
Every subcomplex $S$ of $\Delta^D(\bbN_G)$ can be encoded by an element $f_S \in \{0,1\}^{\Delta^D(\bbN_G)}$
such that $f_S(\sigma) = 1$ exactly if the simplex $\sigma$ is contained in the subcomplex $S$.
We endow  $\{0,1\}^{\Delta^D(\bbN_G)}$ with the product topology, i.e., the topology generated by all cylinder sets.
Let $\mathrm{Sub}(\Delta^D(\bbN_G)) \subseteq \{0,1\}^{\Delta^D(\bbN_G)}$  be the subset which consists of elements encoding $G$-invariant subcomplexes which contain a unique orbit of the form $\{0\}\times X$; this is a closed subspace.

\begin{lemma}
 There is a continuous map $\Psi \colon \SC{G} \to \mathrm{Sub}(\Delta^D(\bbN_G))$ such that $(\Psi([K,o]),\{0\}\times X)$ is a representative of $[K,o]$ for all $[K,o] \in \SC{G}$. 
\end{lemma}
\begin{proof}
  We only sketch the proof; a detailed treatment of the nonequivariant case can be found in \cite[Lemma 1]{schroedl}.
  
   We enumerate $\bbN_G$, the set of
vertices of $\Delta^D(\bbN_G)$, in the following way.  First we enumerate the (isomorphism
classes of) $G$-sets $X_1,\dots, X_k \in \Orb(G)$ and for every $i$ we order the
elements of $X_i = \{x_{i,1},...,x_{i,m_i}\}$. Finally, we enumerate $\bbN_G$ diagonally:
\[(0,x_{1,1}),...,(0,x_{1,m_1}),(0,x_{2,1}),....,(0,x_{k,m_k}),(1,x_{1,1}),...\;.\]
Once the set of vertices is ordered, a diagonal enumeration provides an order on the set of all simplices of $\Delta^D(\bbN_G)$.
The lexicographic order on $\{0,1\}^{\Delta^D(\bbN_G)}$, given by 
 $a\prec b$ if there is a simplex $\sigma_0$ such that $a(\sigma)=b(\sigma)$  for all  $\sigma < \sigma_0$ and  $a(\sigma_0)=1$ but $b(\sigma_0)=0$,
defines an order on the subcomplexes of $\Delta^D(\bbN_G)$. We define $\Psi$ to map an isomorphism class $[K,o]\in\SC{G}$ to the $\prec$-minimal subcomplex $\Lambda$ of $\Delta^D(\bbN_G)$ such that $(\Lambda,\{0\}\times X_i)\in [K,o]$, where $X_i$ is the $G$-set in $\Orb(G)$ with  $X_i\cong o$, and such that the elements of $\{0\}\times X_i$ are the only vertices of $\Lambda$ with first coordinate $0$. For a finite simplicial $G$-complex the existence of the minimal subcomplex of $\Delta^D(\bbN_G)$ follows from the well-ordering principle. In the situation of an infinite simplicial $G$-complex $[K,o]$ it is a direct consequence of the fact that \[B_r(\Psi([B_{r+1}(K,o)]))=\Psi([B_r(K,o)]).\]
The preimage of a cylinder set under $\Psi$ is a countable union of open sets $U(\alpha,r)$ in $\SC{G}$ and therefore open, hence $\Psi$ is continuous. 
\end{proof}

For any simplicial complex $L$, we write $\CC_n(L)$ to denote the complex Hilbert space of square-summable oriented
$n$-chains of $L$.
The map $\Psi$ from the preceding lemma gives rise to a field of Hilbert spaces $[K,o]\to \CC_n(\Psi([K,o]))$ on $\SC{G}$ for each $n\in\bbN$.
In addition, every oriented $n$-simplex $s$ of $\Delta^D(\bbN_G)$ yields a characteristic \emph{vector field} $\xi_s$ defined as
\begin{equation*}
  \xi_s([K,o]) =  \begin{cases}
    s \quad & \text{ if $s$ belongs to $\Psi([K,o])$}\\
    0 & \text{ otherwise.}
                   \end{cases}
\end{equation*}
We observe that, since $\Psi$ is continuous, the function $[K,o]\mapsto \langle\delta_{s}([K,o]),\delta_{s'}([K,o])\rangle$ is continuous for all oriented simplices $s$ and $s'$;
the $\xi_s$ form  a \emph{fundamental sequence} for a \emph{measurable field of Hilbert spaces}; see \cite[Prop.~4]{dixmier}.
A vector field $\sigma\colon[K,o]\mapsto \sigma([K,o])\in\CC_n( \Psi([K,o]))$ is called \emph{measurable}, if
\[[K,o]\mapsto\langle \sigma([K,o]),\xi_s([K,o])\rangle\]
is measurable for every oriented $n$-simplex $s$ of $\Delta^D(\bbN_G)$. Let $\mu$ be a random rooted simplicial $G$-complex. The measurable vector fields $\sigma$ with the property
\[\|\sigma\|^2:=\int_{\SC{G}} \|\sigma([K,o])\|^2 d\mu<\infty\]
form a pre-Hilbert space using
the inner product
\[\langle\sigma,\sigma'\rangle=\int_{\SC{G}}\langle\sigma([K,o]),\sigma'([K,o])\rangle d\mu.\]
By factoring out the subspace of vector fields which vanish almost everywhere, we obtain a Hilbert space: the associated direct integral;
\[\CC_n(\SC{G},\mu) := \int_{\SC{G}}^\oplus \CC_n(\Psi([K,o])) d\mu,\]
for details see \cite[p.~168]{dixmier}.
The differentials $\partial_{n,[K,o]}$ and their adjoints $d_{[K,o]}^n$ of the fibres $\CC_{n}(\Psi([K,o])$ define bounded operators (compare to \cite{schroedl})
\begin{align*}
\partial_n\colon\CC_n(\SC{G},\mu)\to \CC_{n-1}(\SC{G},\mu),\\
d_n\colon \CC_{n-1}(\SC{G},\mu)\to \CC_{n}(\SC{G},\mu),
\end{align*}
which commute with the induced unitary $G$-action on $\CC_n(\SC{G},\mu)$, since they commute fibrewise and $G$ preserves fibres.
Therefore, we have for each random rooted simplicial $G$-complex a chain complex $\CC_*(\SC{G},\mu)$ and a Laplace operator $\Delta_n=\partial_{n+1} \circ d_{n+1} + d_n\circ \partial_n$ which also commutes with the $G$-action. 
\begin{definition}
  We define the \emph{$n$-th simplicial $\ell^2$-homology of a random rooted simplicial complex $\mu$} as the Hilbert space
\[\HH_n(\SC{G},\mu):= \ker \Delta_n\]
equipped with the natural unitary action of $G$.
\end{definition}

We would like to have a notion of dimension for a subspace of $\CC_n(\SC{G},\mu)$, to this end  we introduce a von Neumann algebra with a trace.
A bounded linear operator $T$ on $\CC_n(\SC{G},\mu)$ is \emph{decomposable}, if there is an essentially bounded measurable field of operators $[K,o]\mapsto T_{[K,o]}$ such that $T= \int^{\oplus}T_{[K,o]}d\mu([K,o])$; see \cite[p.~182]{dixmier}.
The bounded decomposable operators $T$ on $\CC_n(\SC{G},\mu)$ such that for almost all $[K,o]$ and for all  isomorphisms $\phi\colon \Psi ([K,o])\to \Psi([K,o'])$ of simplicial $G$-complexes the
identity
\[\langle T_{[K,o]}\sigma([K,o]),\sigma([K,o])\rangle=\langle T_{[K,o']}\phi_\sharp(\sigma([K,o])),\phi_\sharp(\sigma([K,o]))\rangle\]
holds, form a von Neumann algebra $\mathcal{A}_n(\mu)$. In fact, to see that $\mathcal{A}_n(\mu)$ is closed in the strong operator topology one can use \cite[Prop.~4, p.~183]{dixmier}.
Of course, the operators defined by elements of $G$ are contained in $\mathcal{A}_n(\mu)$, since we only consider isomorphisms which commute with the $G$-action. Moreover, $\Delta_n \in \mathcal{A}_n(\mu)$, because $\partial_*$ and $d_*$ commute with the chain map $\phi_\sharp\colon \CC_n(\Psi([K,o]))\to \CC_n(\Psi([K,o']))$ induced by an isomorphism. For $T \in \mathcal{A}_n(\mu)$ we define
\begin{align*}
\tr(T)=\sum_{X\in \Orb(G)}\sum_{x\in X}\sum_{\substack{s\in \Delta^D(\bbN_G)(n)\\ (0,x)\in s}}\frac{\langle T\xi_s,\xi_s\rangle}{|X|(n+1)},
\end{align*}
where $\Delta^D(\bbN_G)(n)$ denotes the set of $n$-simplices of $\Delta^D(\bbN_G)$. Note that the formula does not depend on the chosen orientation of $s$. As in the nonequivariant case one can verify that $\tr(ST)=\tr(TS)$; see \cite{schroedl} after Definition 5. We obtain a normal, faithful and finite trace on $\mathcal{A}_n(\mu)$; compare to \cite[Prop.~3]{schroedl}.
\begin{definition}\label{def:von-Neumann-dim}
 Let $\mathcal{K}$ be a field of $G$-invariant subspaces of $C_n^{(2)}(\SC{G},\mu)$ such that $\phi_\sharp \mathcal{K}([K,o]) = \mathcal{K}([K,o'])$
 for every isomorphism $\phi\colon \Psi([K,o]) \to \Psi([K,o'])$.
 Then the projection $\pr_\mathcal{K}\colon [K,o]\mapsto\pr_{\mathcal{K}([K,o])}$ is an element of
 $\mathcal{A}_n(\mu)$ and we define the \emph{von Neumann dimension} of $\mathcal{K}$ as 
\begin{align*}
\dim_{\mathrm{vN}}(\mathcal{K})=\tr(\pr_\mathcal{K}).
\end{align*}
\end{definition}
\begin{example}\label{ex:total-mass}
  Let $L$ be a finite simplicial $G$-complex and $\mu_L^G$ the associated random rooted simplicial $G$-complex from Example \ref{example-finiterrsc}. Let $\mathcal{K}$ be a field of $G$-invariant subspaces of $C_n^{(2)}(\SC{G},\mu_L^G)$ as in Definition \ref{def:von-Neumann-dim}.
  Given an orbit $o \subseteq L^{(0)}$ and an isomorphism  $\eta\colon L \to \Psi([L,o])$ we define
  $\mathcal{K}(L) = \eta^{-1}(\mathcal{K}([L,o])) \subseteq C_n^{(2)}(L)$; this subspace does not depend on $o$ and $\eta$.
  We compute the dimension:
\begin{align*}
  \dim_{\mathrm{vN}}(\mathcal{K})=\tr(\pr_{\mathcal{K}}) &= \sum_{X\in \Orb(G)}\sum_{x\in X}\sum_{\substack{s\in \Delta^D(\bbN_G)(n)\\ (0,x)\in s}}\int_{\SC{G}}\frac{\langle \pr_{\mathcal{K}}\xi_s,\xi_s\rangle}{|X|(n+1)} d\mu_L^G\\
&=\frac{1}{|L^{(0)}|} \sum_{y \in L^{(0)}}\sum_{x\in Gy}\sum_{\substack{s\in L(n)\\ x \in s}}\frac{\langle \pr_{\mathcal{K}(L)} s, s \rangle}{|Gy|(n+1)}\\
&=\frac{1}{|L^{(0)}|}\sum_{s\in L(n)}\sum_{x\in s}\frac{\langle \pr_{\mathcal{K}(L)} s, s \rangle}{(n+1)}
=\frac{\dim_\bbC \mathcal{K}(L)}{|L^{(0)}|}
\end{align*}
\end{example}

\subsection{Approximation of $\ell^2$-multiplicities}

In order to fix our notation we recall some facts from the representation theory of finite groups.
Let $G$ be a finite group and let $(\rho,V) \in \Irr(G)$ be an irreducible representation of $G$ on the complex vector space $V$. Usually we will not mention the underlying vector space and simply speak of the representation $\rho$.  Every irreducible representation is finite dimensional and is uniquely determined by its character $\chi_\rho\colon G \to \bbC$ which maps $g \in G$ to the trace of $\rho(g)$. In particular, $\chi_\rho(1)$ is the dimension of the underlying space $V$.

Let $(\sigma, W)$ be any finite dimensional complex representation of $G$, then $W$ can be decomposed into isotypic components
\[ W = \bigoplus_{\rho \in \Irr(G)} W_\rho\]
where each $W_\rho$ is (noncanonically) isomorphic to a direct sum of copies of $\rho$, i.e., $W_\rho \cong \rho^{m(\rho,\sigma)}$ where the number $m(\rho,\sigma)$ of copies is called the
\emph{multiplicity} of $\rho$ in~$\sigma$.

The irreducible representations correspond bijectively to the central idempotents in the group ring $\bbC[G]$.
The central idempotent corresponding to $\rho$ is
\[P_\rho=\frac{\chi_\rho(1)}{|G|}\sum_{g\in G}\overline{\chi}_\rho(g)g \in \bbC[G];\] 
see \cite[(2.12)]{Isaacs}. The element $P_\rho$ defines the orthogonal projection onto the $\rho$-isotypic component in every unitary representation of $G$.
\begin{definition}
Let $\mu$ be a random rooted simplicial $G$-complex and $\rho$ an irreducible representation of $G$. 
The \emph{$\ell^2$-multiplicity of $\rho$} in the homology of $\mu$ is 
\[ m^{(2)}_n(\rho,\mu) = \frac{1}{\chi_\rho(1)} \dim_{\mathrm{vN}}\HH_n(\SC{G},\mu)_{\rho} \]
where $\HH(\SC{G},\mu)_{\rho}$ denotes the $\rho$-isotypic component in the homology of $\mu$.

In addition, we define the \emph{$n$-th $\rho$-Laplacian} to be
\[(\id-P_\rho)+\Delta_n=:\Delta_{n,\rho},\]
where $\Delta_n$ is the Laplacian of $\CC_n(\SC{G},\mu)$.
\end{definition}
\begin{remark}
  If $G$ is the trivial group $\{1\}$ and $\rho$ is the unique irreducible representation of $G$, i.e., the
  trivial $1$-dimensional representation, then $m^{(2)}_n(\rho,\mu) = b_n^{(2)}(\mu)$
  is simply the $n$-th \emph{$\ell^2$-Betti number} of a random rooted simplicial complex defined in \cite[Def.~6]{schroedl}.
\end{remark}
\begin{example}\label{ex:finite-complex-multiplicities}
Let $K$ be a finite simplicial $G$-complex and $\mu_K^G$ the associated random rooted simplicial $G$-complex.
Then it follows from Example \ref{ex:total-mass} that $m_n^{(2)}(\rho,\mu_K^G)$ is the ordinary multiplicity of the representation $\rho$ in $H_n(K,\bbC)$ divided by the number of vertices, i.e.,
\[ m_n^{(2)}(\rho,\mu_K^G) = \frac{m(\rho, H_n(K,\bbC))}{|K^{(0)}|}.\]

\end{example}
\begin{lemma}\label{lem:laplace-bounded-selfadjoint}
The operator $\Delta_{n,\rho}$ is positive self-adjoint and the operator norm $\|\Delta_{n,\rho}\|$ is bounded above by a constant $b(n,D)$ which depends only on $n$ and $D$. 
Moreover, the kernel of $\Delta_{n,\rho}$ is $\HH_n(\SC{G},\mu)_\rho$.
\end{lemma}
\begin{proof}
It is easy to see that the operator $\Delta_{n,\rho}$ is positive self-adjoint using that $\Delta_n$ and $\id - P_\rho$ have these properties and commute.
It is well-known (see \cite[Proposition 2]{schroedl}) that the operator norm of the Laplacian $\Delta_n$ is bounded, since the bound on the vertex degree yields a bound for the number of $(n+1)$-simplices which contain a given $n$-simplex. Now, the operator $\id-P_\rho$ is a projection and we find $\| \Delta_{n,\rho} \| \leq \|\Delta_n\| + 1$.

Observe that a vector $x$ lies in $ \ker(\Delta_{n,\rho})$ if and only if $\langle \Delta_{n,\rho} x,\Delta_{n,\rho} x\rangle=0$.
We note further that
\begin{equation*}
\langle \Delta_{n,\rho} x,\Delta_{n,\rho} x\rangle = \|(\id-P_\rho)x\|^2+\|\Delta_n x\|^2 + 2\langle \Delta_n (\id-P_\rho)x, x\rangle
\end{equation*}
since $P_\rho$ and $\Delta_n$ are self-adjoint and commute.
All three summands are nonnegative, since $\Delta_n(\id-P_\rho)$ is a positive operator. We conclude that 
\[\ker(\Delta_{n,\rho}) = \ker \Delta_n\cap \ker(\id-P_\rho) =  \ker\Delta_n \cap \im(P_\rho) =  \HH_n(\SC{G},\mu)_\rho.\]
\end{proof}
Let $E_{\Delta_{n,\rho}}$ be the unique projection valued measure obtained from the spectral calculus for the bounded and self-adjoint operator $\Delta_{n,\rho}$. Then $E_{\Delta_{n,\rho}}$ has the property that for all bounded Borel functions $f$ on $\bbR$
\[f(\Delta_{n,\rho})=\int_{\bbR}f(\lambda) dE_{\Delta_{n,\rho}}(\lambda).\]
Further, we define the \emph{spectral measure of } $\Delta_{n,\rho}$ as
\[\nu_{n,\rho}(B):=\tr ( E_{n,\rho}(B)).\]
for every Borel set $B$. The spectral measure satisfies
$\tr(f(\Delta_{n,\rho}))  = \int_\bbR f(\lambda) d\nu_{n,\rho}$ for all bounded Borel functions $f$ on $\bbR$.

\begin{lemma}\label{lem:weak-convergence}
Let $(\mu^k)_{k\in \bbN}$ be a sequence of random rooted simplicial $G$-complexes which converges weakly to $\mu^\infty$ and let $\nu^k_{n,\rho}$ be the associated spectral measures of the $n$-th $\rho$-Laplacians $\Delta_{n,\rho}$. Then $(\nu^k_{n,\rho})_k$  converges weakly to $\nu^\infty_{n,\rho}$.
\end{lemma}
\begin{proof}
For the sake of simplicity we denote $\nu^k_{n,\rho}$ and $\nu^\infty_{n,\rho}$ by $\nu^k$ and $\nu^\infty$ respectively.
Since, by Lemma \ref{lem:laplace-bounded-selfadjoint}, $\spec(\Delta_{n,\rho})\subseteq [0,R]$ for some $R>0$, Weierstra{\ss} approximation implies that it is enough to check the identity
\[\lim_{k\to\infty}\int_{0}^{R} f(\lambda)d\nu^k=\int_{0}^{R} f(\lambda)d\nu^\infty\]
for all polynomials $f\in \bbR[x]$. By linearity we can further assume that $f=x^r$.
\begin{equation*}
  \int_{0}^{R} f(\lambda)d\nu^k= \tr(\Delta_{n,\rho}^r)
=\sum_{X\in \Orb(G)}\sum_{x\in X}\sum_{\substack{s\in \Delta^D(\bbN_G)(n)\\ (0,x)\in s}}\int_{\SC{G}}\frac{\langle \Delta_{n,\rho}^r\xi_s,\xi_s\rangle}{|X|(n+1)} d\mu^k
\end{equation*}
Let us consider $\Delta_{n,\rho}^r$ and observe that
\begin{align*}
\Delta_{n,\rho}^r=((\id-P_\rho)+\Delta_n)^r&=\sum_{j=0}^r\binom{r}{j}(\id-P_\rho)^{r-j}\Delta_n^j\\
&= \Delta_n^r + \sum_{j=0}^{r-1}\binom{r}{j}\Delta_n^j(\id-P_\rho).
\end{align*}
Let $s\in \Delta(\bbN_G)(n)$ with $(0,x)\in s$ and $x\in X$ and suppose that $s\in\Psi([K,o])$ for a rooted isomorphism class $[K,o]$. Then $(\id-P_\rho)(s)$ is supported in the $1$-ball of the orbit $\{0\}\times X$, since $P_\rho$ is a linear combination of elements $g\in G$ which only act on the second coordinate. Further, $\Delta_n(s)$ is a linear combination of simplices in the $2$-ball around $(0,x)$. Therefore, $\langle \Delta_{n,\rho}^r \xi_s,\xi_s\rangle$ only depends on the $2r+1$-neighbourhood of the orbit $\{0\}\times X$. By the weak convergence of the sequence $(\mu^k)_k$ we obtain that 
\[\lim_{k\to\infty}\int_{U(\alpha,2r+1)}\langle \Delta_{n,\rho}^r\xi_s,\xi_s\rangle d\mu^k=\int_{U(\alpha,2r+1)}\langle \Delta_{n,\rho}^r\xi_s,\xi_s\rangle d\mu^\infty,\]
for all finite rooted simplicial $G$-complexes $\alpha$. Now the claim follows from the fact that $\SC{G}$ is a finite union of open sets of the form $U(\alpha,2r+1)$.
\end{proof}

  Now we can prove Theorem \ref{thm:approximation}. We follow the well-known strategy, going back to L\"uck \cite{Lueck1994} and Schick \cite{Schick2001}, of bounding the Fuglede-Kadison determinant. Here we have to use nonrational algebraic coefficients and the corresponding method is inspired from \cite[Section 3]{Dodziuk-et-al}; compare also \cite[Lemma 3.14]{Kionke18}.
  
  \begin{proof}[Proof of Theorem~\ref{thm:approximation}]
  Since all $\mu_k$ are sofic, it is sufficient to prove the theorem under the assumption that the sequence $\mu_k = \mu_{K_k}^G$ is actually a sequence of finite simplical $G$-complexes.
  Let  $\nu^k_{n,\rho}$ (resp.~$\nu^\infty_{n,\rho}$) be the spectral measure of the $n$-th $\rho$-Laplacian of $K_k$ (resp.\ of $\mu_\infty$).

  By Lemma \ref{lem:laplace-bounded-selfadjoint} we have to show that
  $\nu_{n,\rho}^k(\{0\}) = \chi_\rho(1) m_n^{(2)}(\rho,\mu_{K_k}^G)$ converges to $\nu_{n,\rho}^\infty(\{0\})$.
   In view of Lemma
  \ref{lem:weak-convergence} it remains to show that the \emph{Fuglede-Kadison determinant}
  \[ \det(\nu_{n,\rho}^k) = \exp \int_{\bbR_{>0}} \log(\lambda) d \nu^k_{n,\rho}\]
  of $\nu^k_{n,\rho}$ is uniformly bounded away from zero; see \cite[Lemma 2.20]{Kionke18}.

  Let $E \subseteq \bbC$ be a finite Galois extension of $\bbQ$ which is a splitting field for the finite group $G$; see \cite[(9.10)]{Isaacs} for the existence. In particular, all irreducible characters of $G$ take values only in the ring of integers $\mathcal{O}_E$ of $E$; see \cite[(3.6)]{Isaacs}.

  Fix $k \in \bbN$. We pick a basis of $C_n^{(2)}(K_k)$ by choosing an orientation for every
  $n$-simplex of $K = K_k$.   The transformation matrix $A_\rho$ of the $\rho$-Laplacian
  $\Delta_{n,\rho}$ on $K_k$ with respect to this basis has entries in
  $\frac{1}{|G|}\mathcal{O}_E$.  The spectral measure $\nu^k_{n,\rho}$ agrees with the spectral
  measure of $A_\rho$ normalized by the number of vertices  $|K^{(0)}|$; cf.~Example~\ref{ex:total-mass}.
  In particular, the power $\det(\nu^k_{n,\rho})^{|K^{(0)}|}$ of the Fuglede-Kadison determinant
  is just the product over all non-zero eigenvalues of $A_\rho$, i.e.,
  the lowest non-zero coefficient $c_\rho$ of the characteristic polynomial of $A_\rho$.
  We note that $c_\rho \in |G|^{-|K(n)|} \mathcal{O}_E$ and since, by Lemma \ref{lem:laplace-bounded-selfadjoint}, the operator norm of $\Delta_{n,\rho}$ is bounded above,
  there is an upper bound
  $|c_\rho| \leq b^{|K(n)|}$
  where $b$ depends only on $n$ and $D$. 
  
  Consider the action of the Galois group $\mathrm{Gal}(E/\bbQ)$ on the irreducible representations of $G$; cf.~\cite[p.152]{Isaacs}.
  If $\tau \in \mathrm{Gal}(E/\bbQ)$ be a Galois automorphism of $E$,
  then $\tau(A_\rho) = A_{\tau(\rho)}$ and $\tau(c_\rho) = c_{\tau(\rho)}$.
  In particular, we obtain
  \[ |G|^{ [E:\bbQ] |K(n)|}  \prod_{\tau \in \mathrm{Gal}(E/\bbQ)} \tau(c_{\rho}) \in
    \bbZ \setminus\{0\}\] and therefore
  $|c_{\rho}| \geq |G|^{-[E:\bbQ] |K(n)|} b^{-([E:\bbQ]-1)|K(n)|}$.  Since the vertex degree
  is bounded above by $D$, there is a constant $t > 0$ only depending on $n$ and $D$ such that
  $|K(n)| \leq t |K^{(0)}|$.  We conclude that
  \[ \det(\nu^k_{n,\rho}) = |c_\rho|^{1/|K^{(0)}|} \geq |G|^{-[E:\bbQ]t}b^{-([E:\bbQ]-1)t}\]
  and this lower bound does not depend on $K_k$.
  \end{proof}

Let $G$ be a finite group and $H$ be a subgroup.
Let $(\sigma, V)$ be a finite dimensional complex representation of $H$.
Recall that the multiplicity of an irreducible representation $\rho \in \Irr(G)$
in $\bbC[G]\otimes_{\bbC[H]} V$
is determined by the Forbenius receprocity formula
\[ m(\rho, \bbC[G]\otimes_{\bbC[H]} V) = \sum_{\theta \in \Irr(H)} m(\theta,\rho_{|H}) \; m(\theta, V);\]
see \cite[(5.2)]{Isaacs}. The following theorem provides a reciprocity formula for
induced sofic random rooted $G$-complexes. In particular, with $H = \{1\}$
we obtain Theorem \ref{thm:multiplicities-Betti} stated in the introduction.
\begin{theorem}\label{thm:reciprocity}
  Let $\mu$ be a random rooted simplicial $H$-complex.
  If $\mu$ is sofic, then $\ind_H^G(\mu)$ is sofic and, moreover,
  \begin{equation}\label{eq:reciprocity}
    m^{(2)}_n(\rho,\ind_H^G(\mu)) = \frac{|H|}{|G|} \sum_{\theta\in \Irr(H)} m(\theta,\rho_{|H})\;  m_n^{(2)}(\theta, \mu)
  \end{equation}
  for every $\rho \in \Irr(G)$ and all $n\in \bbN_0$.
\end{theorem}
\begin{proof}
  We write $\nu = \ind_H^G(\mu)$. Since $\mu$ is sofic, we can find a sequence $K_k$ of finite simplicial $H$-complexes such that
  the associated random rooted simplicial complexes $\mu_k$ converge weakly to $\mu$.
  Continuity of induction (Lemma \ref{lem:continuity-of-ind})
  implies that the sequence $\ind_H^G(\mu_k)$ converges to $\nu$. This shows that $\nu$ is
  sofic, since $\ind_H^G(\mu_k)$ is the random rooted simplicial $G$-complex defined by $W_k = G \times_H K_k$.

  Finally, using that $W_k$ is a disjoint union of $|G/H|$ copies of $K_k$ which are permuted by the action
  it follows that
  \[ H_n(W_k,\bbC) \cong \bbC[G] \otimes_{\bbC[H]} H_n(K_k,\bbC).\]
  In particular, Frobenius reciprocity implies that
  \[ m(\rho, H_n(W_k,\bbC)) = \sum_{\theta \in \Irr(H)} m(\theta,\rho_{|H}) \; m(\theta, H_n(K_k,\bbC)).\]
  Based on the relation of multiplicities and $\ell^2$-multiplicities from
  Example \ref{ex:finite-complex-multiplicities} an application of Theorem \ref{thm:approximation} completes the proof.
\end{proof}
\begin{remark}
  It appears to us that formula \eqref{eq:reciprocity} should hold without the assumption of
  soficity.  However, due to some technical problems and based on the fact that we do not know
  a single example of a non-sofic random rooted simplicial complex, we decided to restrict to
  the sofic case.
\end{remark}
\begin{example}(The $\ell^2$-multiplicities of Sierpinski's triangle)\label{ex:Sierpinski2} We
  return to the setting of Example \ref{ex:Sierpinski1} and we compute the
  $\ell^2$-mulitplicities of Sierpinski's triangle with the rotation action of the cyclic
  group $G=\langle \sigma \rangle$ of order $3$.  Recall that, with this action, the
  Sierpinski triangle is an induced random rooted simplicial complex $\ind_1^G(\tau_S)$ where
  $\tau_S$ is the limit of a sequence of finite $2$-dimensional simplicial complexes $T_n$.
  Therefore, by Theorem~\ref{thm:multiplicities-Betti} it is sufficient to compute the $\ell^2$-Betti
  numbers of $\tau_S$.

  In order to use the Approximation Theorem \ref{thm:approximation} for the action of the
  trivial group (compare also \cite{schroedl}), we compute the normalized Betti numbers of
  every $T_n$.  Note that $T_n$ is homotopy equivalent to a $1$-dimensional complex, thus it
  is sufficient to calculate the normalized Euler characteristic of $T_n$.  Using the formulas
  given in Example \ref{ex:Sierpinski1} we find
  \[ \frac{\chi(T_n)}{|T_n^{(0)}|} = 1 - \frac{2 \cdot 3^{n+1}}{3^{n+1}+3} + \frac{2\cdot
      3^n}{3^{n+1}+3} \stackrel{n \to \infty}{\longrightarrow} -\frac{1}{3}.\] Every $T_n$ is
  connected, so $b_0(T_n) = 1$ for all $n$ and the normalized zeroth Betti numbers tend to
  zero; i.e., $b_0^{(2)}(\tau_S) = 0$. We deduce that $b_1^{(2)}(\tau_S) = \frac{1}{3}$.
  Finally, we apply Theorem~\ref{thm:multiplicities-Betti} to deduce that
  $ m^{(2)}_1(\rho,\ind_1^G(\tau_S)) = \frac{1}{9}$ for every irreducible representation
  $\rho \in \Irr(G)$. All $\ell^2$-multiplicities of Sierpinski's triangle vanish outside of
  degree $1$.
\end{example}%
\bibliography{literatur}{}
\bibliographystyle{plain}
\end{document}